\documentclass[11pt]{article}
\usepackage{mathrsfs}
\usepackage{amsmath}
\usepackage{amssymb}
\usepackage{graphicx}
\usepackage{epic}
\usepackage{color}
\usepackage{cases}
\usepackage{pst-poly}  
\usepackage{pst-plot}  
\usepackage{pst-poly}  
\usepackage{CJK}

\renewcommand{\paragraph}{\roman{paragraph}}

\usepackage{bm}
\usepackage{hyperref}
\usepackage{tikz}
\usetikzlibrary{arrows,shapes,positioning}
\usetikzlibrary{decorations.markings}
\tikzstyle arrowstyle=[scale=1]
\tikzstyle directed=[postaction={decorate,decoration={markings, mark=at position .65 with {\arrow[arrowstyle]{stealth}}}}]
\tikzstyle reverse directed=[postaction={decorate,decoration={markings, mark=at position .65 with {\arrowreversed[arrowstyle]{stealth};}}}]

\newtheorem{theorem}{Theorem}[section]
\newtheorem{corollary}[theorem]{Corollary}
\newtheorem{definition}[theorem]{Definition}

\newtheorem{lemma}[theorem]{Lemma}

\newenvironment{proof}{\noindent {\bf Proof.}}{\rule{3mm}{3mm}\par\medskip}

\begin{document}

\title{On extremal spectral radius of blow-up uniform hypergraphs\thanks{Supported by National Natural Science Foundation of China (12171002,
11871073) and Anhui Provincial Natural Science Foundation (No. 2108085MA02).}}
\author{Shao-Han Xu, Fu-Tao Hu\footnote{Corresponding author:
hufu@ahu.edu.cn.}, Yi Wang\\
{\small Center for Pure Mathematics, School of mathematical Sciences, Anhui University, Hefei 230601, China}
}
\date{}
\maketitle
\begin{abstract}
Let $G$ be an $r$-uniform hypergraph of order $t$ and $\rho(G)$ is the spectral radius of $\mathcal{A}(G)$,
where $\mathcal{A}(G)$ is the adjacency tensor of $G$. A blow-up of $G$ respected to a positive integer vector  $(n_{1}, n_{2},\ldots,n_{t})$, denoted by $G \circ (n_{1}, n_{2},\ldots,n_{t})$,  is an $r$-uniform hypergraph obtained from $G$ by replacing each vertex $j$ of $G$
with a class of vertices $V_{j}$ of size $n_{j}\ge 1$ and if $\{j_{1},j_{2},\ldots,j_{r}\}\in E(G)$,
then $\{v_{i_1},v_{i_2},\ldots,v_{i_r}\}\in E(H)$ for every $v_{i_{1}}\in V_{j_{1}}, v_{i_{2}}\in V_{j_{2}},\ldots, v_{i_{r}}\in V_{j_{r}}$.
Let  $\mathcal{B}_{n}(G)$ be the set of all the blow-ups of $G$ such that each $n_i\ge 1$ and $\sum_{i=1}^n n_i=n$.
Let $K_{t}^{r}$  be the complete $r$-uniform hypergraph of order $t$, and let $SH(m,q,r)$ be the $r$-uniform sunflower hypergraph with $m$ petals and a kernel of size $r-q$ on $t$ vertices. For any $H\in \mathcal{B}_{n}(K_{t}^{r})$, we prove that
$$\rho(K_{t}^{r}\circ(n-t+1,1,1,\ldots,1))\leq\rho(H)\leq \rho (T_{t}^{r}(n)),$$
with the left equality holds if and only if $H\cong K_{t}^{r}\circ(n-t+1,1,1,\ldots,1)$, and the right equality holds if and only if  $H\cong T_{t}^{r}(n)$, where $T_{t}^{r}(n)$ is the complete $t$-partite $r$-uniform hypergraph of order $n$, with parts of size $\lfloor n / k\rfloor$ or $\lceil n / k \rceil$. For any $H\in \mathcal{B}_{n}(H(m,q,r))$, we determine the exact value of the spectral radius of $H$
and characterize the hypergraphs with maximum spectral radius and minimum spectral radius in $\mathcal{B}_{n}(H(m,q,r))$, respectively.
\end{abstract}



\section{Introduction}

In recent years, the study of spectral hypergraph theory via tensors has attracted extensive attention.
An important topic in spectral hypergraph theory is to characterize hypergraphs with extremal
spectral radius for a given class of hypergraphs.
In this paper, we consider two classes of blow-ups of $r$-uniform hypergraphs.

Let $G$ be an $r$-uniform hypergraph with vertex set $V(G)=[t]:=\{1,2,\ldots,t\}$,
and let $(n_{1}, n_{2},\ldots,n_{t})$ be a positive integer vector. The {\it blow-up} of $G$,
denoted by $G\circ(n_{1}, n_{2},\ldots,n_{t})$, is an $r$-uniform hypergraph
obtained from $G$ by replacing each vertex $j\in [t]$ with a class of vertices $V_{j}$ of size $n_{j}\ge 1$ and
if $\{j_{1},j_{2},\ldots,j_{r}\}\in E(G)$, then $\{v_{i_1},v_{i_2},\ldots,v_{i_r}\}\in E(G\circ(n_{1}, n_{2},\ldots,n_{t}))$
for every $v_{i_{1}}\in V_{j_{1}}, v_{i_{2}}\in V_{j_{2}},\ldots, v_{i_{r}}\in V_{j_{r}}$.
Let   $\mathcal{B}_{n}(G)$ be the set of $r$-uniform hypergraphs $G\circ(n_{1}, n_{2},\ldots,n_{t})$,
where $(n_{1}, n_{2},\ldots,n_{t})$ takes over all positive integer vectors with $\sum_{i=1}^{t}n_{i}=n$.

Denote by $K_{t}$ the complete graph of order $t$. Stevanovi\'{c},
Gutman and Rehman \cite{D2} determined the maximum spectral radius and the minimum spectral
radius over all graphs in $\mathcal{B}_{n}(K_{t})$. The corresponding extremal graphs are the $t$-bipartite Tur\'{a}n graph $T_{t}(n)=K_{t}\circ(\lceil\frac{n}{t}\rceil,\lceil\frac{n}{t}\rceil,\ldots,\lfloor\frac{n}{t}\rfloor,\lfloor\frac{n}{t}\rfloor)$ and $K_{t}\circ(n-t+1,1,1,\ldots,1)$, respectively. Recently, Lou and Zhai \cite{Z}, Monsalve and Rada \cite{J1},
Sun and Das \cite{S}, Zhai  et al. \cite{M} studied the extremal spectral radius problem related to blow-up graphs.

The spectral radius of an $r$-uniform hypergraph refers to the spectral radius of its adjacency tensor.
In this paper, we consider the extremal spectral radius problem for $r$-uniform hypergraphs.
An $r$-uniform hypergraph on $t$ vertices having all possible edges is called a {\it complete
$r$-uniform hypergraph} and is denoted by $K_{t}^{r}$. An $r$-uniform hypergraph is called
{\it $t$-partite} if its vertex set can be  partitioned into $t$ sets so that each
edge contains at most one vertex from each set. An edge maximum $t$-partite $r$-uniform hypergraph
is called {\it complete  $t$-partite}.  We write $T_{t}^{r}(n)$ for the complete $t$-partite $r$-uniform
hypergraph of order $n$, with parts of size $\lfloor\frac{n}{t}\rfloor$ or $\lceil\frac{n}{t}\rceil$.
Note that  $T_{t}^{2}(n)$ is the {\it Tur\'{a}n graph} $T_{t}(n)$. Obviously, a complete $t$-partite
$r$-uniform hypergraph is a blow-up of a complete $r$-uniform hypergraph $K_{t}^{r}$. In  \cite{L},
Kang, Nikiforov and Yuan determined the $r$-uniform hypergraph with maximum $p$-spectral radius in all
$t$-partite $r$-uniform hypergraphs of order $n$ is $T_{t}^{r}(n)$. This indicates that the maximum
spectral radius in $\mathcal{B}_{n}(K_{t}^{r})$ when $p=r$. In Section $3$, we determine the minimum spectral
radius over all hypergraphs  in $\mathcal{B}_{n}(K_{t}^{r})$, which generalizes the result of Stevanovi\'{c},
Gutman and  Rehman in \cite{D2} for simple graphs.

Let $m>0$, and $q,r$ satisfy $0<q<r$. A {\it sunflower} hypergraph $SH(m,q,r)$ \cite{c15} is defined as follows.
Let $X$ be a set of $r-q$ vertices (``kernel") and define $m$ disjoint sets $\{Y_{i}\}_{i=1}^{m}$ of $q$ vertices each (``petals").
The edges of the hypergraph are the sets $X\cup Y_{i}$ for  $1\leq i\leq m$.
When $r=2$, sunflower graphs are normally referred to as {\it stars}.
When $q=1$, $S (m, 1, r)$ is a complete {\it $r$-cylinder} (i.e. complete $r$-partite, $r$-uniform hypergraph)
with parts sizes $m$ and $1^{k-1}$, which the authors considered in~\cite{J}.
When $q=r-1$, $S (m, r-1, r)$ is a {\it hyperstar} which the authors considered in~\cite{hqs13}.

Several studies on sunflower hypergraphs can be seen in \cite{DDD,III,II}.  Let $G$ be an $r$-uniform hypergraph, $p\ge 1$ and $k\ge 1$.
Nikiforov in~\cite{V} gave the exact value of the $p$-spectral radius of $G\circ (k,k,\ldots,k)$, that is $\lambda^{(p)}(G\circ (k,k,\ldots,k))=k^{r-r/p}\lambda^{(p)}(G)$. In Section $4$, we study the spectral radius of blow-ups of $SH(m,q,r)$.
We give the exact value of the spectral radius of $SH(m,q,r)\circ(n_{1}, n_{2},\ldots,n_{t})$
and characterize the hypergraphs with maximum spectral radius and minimum spectral radius in $\mathcal{B}_{n}(SH(m,q,r))$.

\section{Preliminaries}

\subsection{Eigenvalues of tensors}

A {\it real tensor} (also called {\it hypermatrix}) $\mathcal{T}=(t_{i_{1}i_{2}\cdots i_{r}})$ of order $r$ and dimension $n$ refers to a multi-dimensional array with entries $t_{i_{1}i_{2}\cdots i_{r}}\in\mathbb{R}$ for all $i_{j}\in [n]$ and $j\in[r]$. Obviously, if $r=2$, then $\mathcal{T}$ is a square matrix of dimension $n$.  The tensor $\mathcal{T}$ is called {\it symmetric} if its entries are invariant under any permutation of their indices. Given a vector $\mathbf{x}\in \mathbb{C}^{n}$, $\mathcal{T} \mathbf{x}^{r}\in\mathbb{C}$ and  $\mathcal{T}\mathbf{x}^{r-1} \in\mathbb{C}^{n}$, which are defined as follows:

$$\mathcal{T}\mathbf{x}^{r}=\sum_{i_{1},i_{2},\ldots,i_{r}\in [n]}t_{i_{1}i_{2}\cdots i_{r}}x_{i_{1}}x_{i_{2}}\cdots x_{i_{r}},$$
$$(\mathcal{T}\mathbf{x}^{r-1})_{i}=\sum_{i_{2},\ldots,i_{r}\in [n]}t_{ii_{2}\cdots i_{r}}x_{i_{2}}\cdots x_{i_{r}}, \;i\in [n].$$
Let $\mathcal{I}=(i_{i_{1}i_{2}\cdots i_{r}})$ be the identity tensor of order $r$ and dimension $n$, that is, $i_{i_{1}i_{2}\cdots i_{r}}=1$ if  $i_{1}=i_{2}=\cdots=i_{r}\in[n]$ and $i_{i_{1}i_{2}\cdots i_{r}}=0$ otherwise.
In 2005, Qi~\cite{L1} and Lim~\cite{Li} independently introduced the concept of eigenvalues for tensors as follows.

\begin{definition}
Let $\mathcal{T}$ be a real tensor of order $r$ dimension $n$. For some $\lambda\in\mathbb{C}$, if the polynomial system $(\lambda\mathcal{I}-\mathcal{T})\mathbf{x}^{r-1}=0$, or equivalently $\mathcal{T}\mathbf{x}^{r-1}=\lambda \mathbf{x}^{[r-1]}$, has a solution $\mathbf{x}\in \mathbb{C}^{n}\backslash\{0\}$, then $\lambda$ is called an eigenvalue of $\mathcal{T}$ and $\mathbf{x}$ is an eigenvector of $\mathcal{T}$ associated with $\lambda$, where $\mathbf{x}^{[r-1]}:=(x_{1}^{r-1},x_{2}^{r-1},\ldots,x_{n}^{r-1})\in\mathbb{C}^{n}$.
\end{definition}

The {\it spectral radius} of $\mathcal{T}$ is defined as $\rho(\mathcal{T})=\max\{|\lambda|:\lambda\;\;\rm is\;\; an\;\; eigenvalue\;\;of\;\;\mathcal{T}\}$. Let $\mathbb{R}_{+}^{n}=\{\mathbf{x}\in \mathbb{R}^{n}:~ \mathbf{x}\geq 0\}$. For the  spectral radius of  a  symmetric nonnegative tensor, we have the following lemma.

\begin{lemma}\label{1} (\cite{SSS,Qiii})
Let $\mathcal{T}$ be a symmetric nonnegative tensor of order $r$ and dimension $n$. Then
\begin{displaymath}
\rho(\mathcal{T})=\max\{\mathcal{T}\mathbf{x}^{r}:\mathbf{x}\in \mathbb{R}_{+}^{n}, \sum_{i=1}^{n} x _{i}^{r}=1\}.
 \end{displaymath}
Furthermore, $\mathbf{x}\in \mathbb{R}_{+}^{n}$ with $\sum_{i=1}^{n}x _{i}^{r}=1$ is an eigenvector of $\mathcal{T}$ corresponding to $\rho(\mathcal{T})$ if and only if it is an optimal solution to the above maximization problem.
\end{lemma}

\subsection{Uniform hypergraphs}

A {\it hypergraph} $H=(V(H),E(H))$ is a pair consisting of a vertex set $V(H)=\{v_{1},v_{2},\ldots,v_{n}\}$ and an edge set $E(G)=\{e_{1},e_{2},\ldots,e_{m}\}$, where $e_{j}\subseteq V(H)$ for each $j\in [m]$. If $|e_{j}|=r$ for all $j\in [m]$,
then $H$ is called an {\it $r$-uniform hypergraph}. In particular,  a $2$-uniform hypergraph is a graph.
Two vertices are said to be {\it adjacent} in $H$ if there is an edge that contains both of them.
A hypergraph $H'$ is a {\it subhypergraph} of $H$ if $V(H') \subseteq V(H)$ and $E(H')\subseteq E(H)$.
For a subset $S\subseteq V(H)$, we denoted by $E_{H}(S)$ the set of edges $\{e\in E(H) : S\cap e\neq \emptyset\}$.
For a vertex $v\in V(H)$, we simplify $E_{H}(\{v\})$ as $E_{H}(v)$.
For a subset $S\subseteq V(H)$, we define $H-S$ to be the subhypergraph obtained from $H$
by deleting the vertices in $S$ and all edges in $E_H(S)$. If $S=\{v\}$, then we write  $H-S$ simple as $H-v$.
And for a subset $E'\subseteq E(H)$,  we define  $H-E'$ to be the hypergraph obtained from $H$ by deleting the edges in $E'$ and resulting isolated vertices. If $E'=\{e\}$, then we write  $H-E'$ simple as $H-e$.

Let $H$ be an $r$-uniform hypergraph with vertex set $V(H)=\{v_{1},v_{2},\ldots,v_{n}\}$.
The {\it adjacency tensor} of $H$ defined by Cooper and Dutle \cite{J} as $\mathcal{A}(H)=(a_{i_{1}i_{2}\cdots i_{r}})$, is
an $r$th order and $n$-dimensional tensor, where
\begin{displaymath}
a_{i_{1}i_{2}\cdots i_{r}}=
   \begin{cases}
   \frac{1}{(r-1)!}, &\mbox {\rm if $\{v_{i_{1}},v_{i_{2}},\ldots,v_{i_{r}}\}\in E(H)$};\\
   \;\;\;\;0, &\mbox {\rm otherwise}.
   \end{cases}
\end{displaymath}

Eigenvalues of $\mathcal{A}(H)$ are called eigenvalues of $H$, the spectral radius of $\mathcal{A}(H)$ is called the spectral radius of $H$,
denoted by $\rho(H)$. By the Perron-Frobenius theorem for nonnegative tensors  \cite{Ch,Fr,Y},  $\rho(H)$ is exactly the largest eigenvalue of $\mathcal{A}(H)$. If $H$ is connected, then there exists a unique positive eigenvector up to a multiplicative constant corresponding to $\rho(H)$, called the {\it Perron vector} of $H$.

Let $\mathbf{x}=(x_{1},x_{2},\ldots,x_{n})\in\mathbb{C}^{n}$. Then $\mathbf{x}$ can be considered as a function defined on the vertices of $H$, that is, each vertex $v_{i}$ is mapped to $x_{i}:=x_{v_{i}}$. If $\mathbf{x}$ is an eigenvector of $\mathcal{A}(H)$, then it defines on $H$ naturally, i.e., $x_{v}$ is the entry of $\mathbf{x}$ corresponding to $v$. For a subset $U$ of $V(H)$, denote $x^{U}=\prod\limits_{u\in U}x_{u}$. Then

\begin{displaymath}\tag{1}
\mathcal{A}(H)\mathbf{x}^{r}=\sum_{e\in E(H)}rx^{e}.
 \end{displaymath}
And the eigenvector  equation  $\mathcal{A}(H)\mathbf{x}^{r-1}=\lambda \mathbf{x}^{[r-1]}$  is equivalent to that for each $v\in V(H)$,
\begin{displaymath}\tag{2}
 \begin{split}
\lambda x^{r-1}_{v}
=\sum_{e\in E_{H}(v)}x^{e\setminus \{v\}}.
 \end{split}
\end{displaymath}

The following result is useful for proofs.

\begin{lemma}\label{2}(\cite{J,Khan})
Let $H$ be a connected $r$-uniform hypergraph, and let $H'$ be a subhypergraph of $H$. Then $\rho(H')\leq\rho(H)$ with equality if and only if $H'=H$.
\end{lemma}

\section{Extremal spectral radius of hypergraphs in $\mathcal{B}_{n}(K_{t}^{r})$}

Let $G$ be a connected $r$-uniform hypergraph with vertex set $V(G)=[t]$.
Let $L_{G}(i)=\{e\setminus \{i\}: e\in E_{H}(i)\}$ for any vertex $i\in V(G)$.
For any vertex $i\in V(G)$ and nonempty subset $S\subseteq V(G)\setminus \{i\}$,
we have $L_{G-S}(i)=\{e\setminus \{i\}: e\in E_{G}(i)\setminus E_G(S)\}$ (write $L_{G-j}(i)=L_{G-\{j\}}(i)$ for simplicity).
In order to find the hypergraphs with maximum spectral radius and minimum spectral radius in $\mathcal{B}_{n}(K_{t}^{r})$, we need the following key lemma.

\begin{lemma}\label{4}
Let $G$ be a connected  $r$-uniform hypergraph of order $t$,  and let $(n_{1}, n_{2},\ldots,n_{t})$ be a vector of positive integers
with $\sum_{i=1}^t n_i=n$. Let $i,j$ be two adjacent vertices in $G$. If $L_{G-j}(i)\subseteq L_{G-i}(j)$ and $n_{i}-n_{j}\geq 2$
(without loss of generality, assume $i<j$), then
\begin{displaymath}
\begin{split}
\rho(G\circ(n_{1},\ldots,n_{i}-1,\ldots,n_{j}+1,\ldots,n_{t}))>\rho(G\circ(n_{1},\ldots,n_{i},\ldots,n_{j},\ldots,n_{t})).
\end{split}
\end{displaymath}
\end{lemma}
\begin{proof}
Let $G_{1}=G\circ(n_{1},\ldots,n_{i},\ldots,n_{j},\ldots,n_{t})$ and $G_{2}=G\circ({n_{1},\ldots,n_{i}-1,\ldots,n_{j}+1,\ldots,n_{t}})$.
Suppose $V(G)=[t]$. By definition of blow-up hypergraphs,  $G_{1}$ has vertex partition: $V(G_{1})=V_{1}\cup\cdots\cup V_{i}\cup \cdots\cup V_{j}\cup \cdots \cup V_{t}$,  where $|V_{k}|=n_{k}$ for each $k\in [t]$, and if $\{j_{1},j_{2},\ldots,j_{r}\}\in E(G)$, then $\{v_{i_1},v_{i_2},\ldots,v_{i_r}\}\in E(G_{1})$ for every $v_{i_{1}}\in V_{j_{1}}, v_{i_{2}}\in V_{j_{2}},\ldots, v_{i_{r}}\in V_{j_{r}}$.

Fix two vertices $v_i\in V_i$ and $v_j\in V_j$. Note that $G_{2}$ is a hypergraph can be obtained from $G_{1}$
by deleting every edge that contains $v_i$ and adding all edges in $\{f\cup \{v_{i}\}: f\in L_{G_{1}-v_{i}}(v_{j})\}$.
Hence, $G_{2}$ has vertex partition: $V(G_{2})=V_{1}'\cup\cdots \cup V_{i}'\cup\cdots\cup V_{j}'\cup\cdots \cup V_{t}'$, where $V_{i}'=V_{i}\setminus\{v_{i}\}, \;V_{j}'=V_{j}\cup \{v_{i}\}$  and $V_{k}'=V_{k}$ for each $k\neq i,j$.

Let $\mathbf{x}$ be the Perron vector of $G_{1}$ with $\sum_{v\in V(G_1)} x_v^r=1$.  For any two distinct vertices
$u,v\in V_{k},~k\in [t]$,
note that $L_{G_1}(u)=L_{G_1}(v)$, we have $x_{u}=x_{v}$ by definition of eigenvector equation.
For each $k\in [t]$, we write $x_{u}=x_{k}$ for any vertex $u\in V_{k}$.
Choose any vertices $u\in V_i$, $v\in V_j$, $a\in V_i'$ and $b\in V_j'$. By definition of blow-up hypergraphs, we simply write
$$
S_{1}=\sum\limits_{f:\,\{u,v\}\cup f \in E(G_1)} x^f=\sum\limits_{f:\,\{a,b\}\cup f \in E(G_2)} x^f=\sum\limits_{\{i,j,i_{1},\ldots, i_{r-2}\}\in E(G)}n_{i_{1}}x_{i_{1}}\cdots n_{i_{r-2}}x_{i_{r-2}}~(r\ge 3),
$$
and set $S_1=1$ for $r=2$. Since $i$ and $j$ are adjacent in $G$, $S_1>0$.
Note that $L_{G_1-V_j}(u)=L_{G_2-V_j'}(a)$ and $L_{G_1-V_i}(v)=L_{G_2-V_i'}(b)$. By definition of blow-up hypergraphs, we simply write
$$S_{2}'=\sum\limits_{g\in L_{G_1-V_j}(u)} x^g=\sum\limits_{g\in L_{G_2-V_j'}(a)} x^g=\sum\limits_{\{i_{1},\ldots,i_{r-1}\}\in L_{G-j}(i) } n_{i_{1}}x_{i_{1}}\cdots n_{i_{r-1}}x_{i_{r-1}}.$$
and
$$S_{2}''=\sum\limits_{g\in L_{G_1-V_i}(v)} x^g=\sum\limits_{g\in L_{G_2-V_i'}(b)} x^g=\sum\limits_{\{i_{1},\ldots,i_{r-1}\}\in L_{G-i}(j) } n_{i_{1}}x_{i_{1}}\cdots n_{i_{r-1}}x_{i_{r-1}}.$$
We have $S_{2}''\geq S_{2}'$ as $L_{G-j}(i)\subseteq L_{G-i}(j)$.
By  Lemma \ref{1} and equation $(1)$, we have
\begin{displaymath}\tag{3}
 \begin{split}
 \rho(G_{2})-\rho(G_{1})&\geq\mathcal{A}(G_{2})\mathbf{x}^{r}-\mathcal{A}(G_{1})\mathbf{x}^{r}\\
 &=\sum_{e\in E(G_{2})}rx^{e}-\sum_{e\in E(G_{1})}rx^{e}\\
 &=\sum_{e\in E_{G_{2}}(v_i)} rx^{e}-\sum_{e\in E_{G_{1}}(v_i)} rx^{e}.
  \end{split}
\end{displaymath}
Since
\begin{displaymath}
 \begin{split}
 \sum_{e\in E_{G_{2}}(v_i)} rx^{e}=&\sum_{e\in E_{G_{2}}(v_{i})\setminus E_{G_{2}}(V_{i}')}rx^{e}+\sum_{e\in E_{G_{2}}(v_{i})\cap E_{G_{2}}(V_{i}')}rx^{e}\\
=&rx_{v_i} \sum_{g\in L_{G_2-V_i'}(v_{i}) } x^g+r  x_{v_i}\cdot\sum_{a\in V_i'} x_{a} \sum_{f:\,\{a,v_i\}\cup f\in E(G_2)} x^f\\
=&rx_{i}\cdot S_{2}''+rx_i \cdot(n_i-1)x_i\cdot S_{1},
\end{split}
\end{displaymath}
and similarly
\begin{displaymath}
 \begin{split}
 \sum_{e\in E_{G_{1}}(v_i)} rx^{e}=&\sum_{e\in E_{G_{1}}(v_{i})\setminus E_{G_{1}}(V_{j})}rx^{e}+\sum_{e\in E_{G_{1}}(v_{i})\cap E_{G_{1}}(V_{j})}rx^{e}\\
=&rx_{i}\cdot S_{2}'+rx_i  \cdot n_jx_j\cdot S_{1},
\end{split}
\end{displaymath}
by (3), we have
\begin{displaymath}
 \begin{split}
&\rho(G_{2})-\rho(G_{1})\\
\geq&  rx_{i}\cdot (S_{2}''-S_{2}')+rx_i  \cdot S_{1}\cdot[(n_i-1)x_i-n_jx_j].
\end{split}
\end{displaymath}
If $(n_i-1)x_i-n_j x_j>0$, then $\rho(G_{2})>\rho(G_{1})$.

In the following, we assume $(n_i-1)x_i-n_j x_j\leq 0 $. Define a vector $\mathbf{y}$ for $G_{2}$ as follows:
\begin{displaymath}
y_{v}=
   \begin{cases}
   \left(\frac{n_{i}x_{i}^{r}}{n_{i}-1}\right)^{\frac{1}{r}}, &\mbox {if $v\in V_{i}'$},\\
    \left(\frac{n_{j}x_{j}^{r}}{n_{j}+1}\right)^{\frac{1}{r}}, &\mbox {if $v\in V_{j}'$},\\
     \;\;x_{k}, &\mbox {if $v\in V_{k}'$}\;{\rm and} \;k\neq i,j.
   \end{cases}
\end{displaymath}
It is easy to check that $\sum\limits_{v\in V(G_2)} y_v^r=1$. By Lemma \ref{1} and  equation (1),
we have
\begin{displaymath}\tag{4}
 \begin{split}
\rho(G_{2})-\rho(G_{1})\geq&\mathcal{A}(G_{2})\mathbf{y}^{r}-\mathcal{A}(G_{1})\mathbf{x}^{r}\\
=&\sum_{e\in E(G_{2})}ry^{e}-\sum_{e\in E(G_{1})}rx^{e}\\
=&r\sum_{e\in E_{G_{2}}(V_{i}'\cup V_{j}' )} y^{e}-r\sum_{e\in E_{G_{1}}(V_{i}\cup V_{j})} x^{e}.
\end{split}
\end{displaymath}

Since
\begin{displaymath}
 \begin{split}
&\sum_{e\in E_{G_{2}}(V_{i}'\cup V_{j}' )} y^{e}\\
=&\sum_{e\in E_{G_{2}}(V_{i}')\setminus E_{G_{2}}(V_{j}' )} y^{e}+\sum_{e\in E_{G_{2}}(V_{j}')\setminus E_{G_{2}}(V_{i}' )} y^{e}+\sum_{e\in E_{G_{2}}(V_{i}')\cap E_{G_{2}}(V_{j}' )} y^{e}\\
=&\sum\limits_{v_{i}'\in V_{i}'} \sum_{e\in E_{G_{2}}(v_{i}')\setminus E_{G_{2}}(V_{j}' )} y^{e}+\sum\limits_{v_{j}'\in V_{j}'} \sum_{e\in E_{G_{2}}(v_{j}')\setminus E_{G_{2}}(V_{i}' )} y^{e}+\sum\limits_{v_{i}'\in V_{i}'} y_{v_{i}'}\cdot\sum\limits_{v_{j}'\in V_{j}'} y_{v_{j}'} \sum\limits_{f:\,\{v_i',v_j'\}\cup f\in E(G_2)} y^f\\
=&\sum\limits_{v_{i}'\in V_{i}'} y_{v_i'} \sum_{g\in L_{G_2-V_j'}(v_i')}  x^{g}+\sum\limits_{v_{j}'\in V_{j}'} y_{v_j'}\sum_{g\in L_{G_2-V_i'}(v_j')} x^{g}+\sum\limits_{v_{i}'\in V_{i}'} y_{v_{i}'}\cdot\sum\limits_{v_{j}'\in V_{j}'} y_{v_{j}'} \sum\limits_{f:\,\{v_i',v_j'\}\cup f\in E(G_2)} x^f\\
=& (\sum\limits_{v_{i}'\in V_{i}'} y_{v_{i}'})\cdot S_{2}'+ (\sum\limits_{v_{j}'\in V_{j}'} y_{v_{j}'})\cdot S_{2}''+ (\sum\limits_{v_{i}'\in V_{i}'} y_{v_{i}'}\cdot\sum\limits_{v_{j}'\in V_{j}'} y_{v_{j}'})\cdot S_{1}\\
=&S_{2}'\cdot(n_{i}-1)\left(\frac{n_{i}x_{i}^{r}}{n_{i}-1}\right)^{\frac{1}{r}}+S_{2}''\cdot(n_{j}+1)\left(\frac{n_{j}x_{j}^{r}}{n_{j}+1}\right)^{\frac{1}{r}}
+S_{1}\cdot(n_{i}-1)\left(\frac{n_{i}x_{i}^{r}}{n_{i}-1}\right)^{\frac{1}{r}}(n_{j}+1)\left(\frac{n_{j}x_{j}^{r}}{n_{j}+1}\right)^{\frac{1}{r}},
 \end{split}
\end{displaymath}
and similarly
\begin{displaymath}
 \begin{split}
\sum_{e\in E_{G_{1}}(V_{i}\cup V_{j})} x^{e}
=&\sum_{e\in E_{G_{1}}(V_{i})\setminus E_{G_{1}}(V_{j} )} x^{e}+\sum_{e\in E_{G_{1}}(V_{j})\setminus E_{G_{1}}(V_{i})} x^{e}+\sum_{e\in E_{G_{1}}(V_{i})\cap E_{G_{1}}(V_{j})} x^{e}\\
=& (\sum\limits_{v_{i}\in V_{i}}x_{v_{i}})\cdot S_{2}'+ (\sum\limits_{v_{j}\in V_{j}}x_{v_{j}})\cdot S_{2}+ (\sum\limits_{v_{i}\in V_{i}}x_{v_{i}}\cdot\sum\limits_{v_{j}\in V_{j}}x_{v_{j}})\cdot S_{1}\\
=&  S_{2}'\cdot n_{i}x_{i}+S_{2}''\cdot n_{j}x_{j}+  S_{1}\cdot n_{i}x_{i}n_{j}x_{j}.
 \end{split}
\end{displaymath}
by (4), we have
\begin{displaymath}\tag{5}
 \begin{split}
&\rho(G_{2})-\rho(G_{1})\\
\geq&  r S_{1}\cdot \left[(n_{i}-1)\left(\frac{n_{i}x_{i}^{r}}{n_{i}-1}\right)^{\frac{1}{r}}(n_{j}+1)\left(\frac{n_{j}x_{j}^{r}}{n_{j}+1}\right)^{\frac{1}{r}}-n_{i}x_{i}n_{j}x_{j}\right]\\
&+r S_{2}'\cdot \left[(n_{i}-1)\left(\frac{n_{i}x_{i}^{r}}{n_{i}-1}\right)^{\frac{1}{r}}-n_{i}x_{i}\right]+rS_{2}''\cdot\left[(n_{j}+1)\left(\frac{n_{j}x_{j}^{r}}{n_{j}+1}\right)^{\frac{1}{r}}-n_{j}x_{j}\right]\\
 =&r S_{1}\cdot x_{i}x_{j}\cdot \left[(n_{i}-1)^{1-\frac{1}{r}}(n_{j}+1)^{1-\frac{1}{r}}n_{i}^{\frac{1}{r}}n_{j}^{\frac{1}{r}}-n_{i}n_{j}\right]\\
&+r S_{2}'\cdot \left[(n_{i}-1)^{1-\frac{1}{r}}n_{i}^{\frac{1}{r}}-n_{i}\right]x_{i}+rS_{2}''\left[(n_{j}+1)^{1-\frac{1}{r}}n_{j}^{\frac{1}{r}}-n_{j}\right]x_{j}\\
\geq &r S_{1}\cdot x_{i}x_{j}\cdot \left[(n_{i}-1)^{1-\frac{1}{r}}(n_{j}+1)^{1-\frac{1}{r}}n_{i}^{\frac{1}{r}}n_{j}^{\frac{1}{r}}-n_{i}n_{j}\right]\\
&+r S_{2}'\cdot \left[\left((n_{i}-1)^{1-\frac{1}{r}}n_{i}^{\frac{1}{r}}-n_{i}\right)x_{i}+\left((n_{j}+1)^{1-\frac{1}{r}}n_{j}^{\frac{1}{r}}-n_{j}\right)x_{j}\right].
\end{split}
\end{displaymath}

Since  $n_{i}-n_{j}\geq 2$, we have $(n_{i}-1)(n_{j}+1)=n_{i}n_{j}+n_{i}-n_{j}-1>n_{i}n_{j}$. Applying this fact to (5),
it suffices to prove that
\begin{displaymath}\tag{6}
\left((n_{i}-1)^{1-\frac{1}{r}}n_{i}^{\frac{1}{r}}-n_{i}\right)x_{i}+\left((n_{j}+1)^{1-\frac{1}{r}}n_{j}^{\frac{1}{r}}-n_{j}\right)x_{j}>0,
\end{displaymath}
which will imply $\rho(G_{2})>\rho(G_{1})$ by (5).

Since  $(n_i-1)x_i-n_j x_j\leq0$, we have $x_{i}\leq \frac{n_{j}}{n_{i}-1}x_{j}$.  Note that  $(n_{i}-1)^{1-\frac{1}{r}}n_{i}^{\frac{1}{r}}-n_{i}<0$ and $n_{i}-1\geq n_{j}+1$.  Then
\begin{displaymath}\tag{7}
 \begin{split}
&\left[(n_{i}-1)^{1-\frac{1}{r}}n_{i}^{\frac{1}{r}}-n_{i}\right]x_{i}+\left[(n_{j}+1)^{1-\frac{1}{r}}n_{j}^{\frac{1}{r}}-n_{j}\right]x_{j}\\
\geq&\left[(n_{i}-1)^{1-\frac{1}{r}}n_{i}^{\frac{1}{r}}-n_{i}\right]\frac{n_{j}}{n_{i}-1}x_{j}+\left[(n_{j}+1)^{1-\frac{1}{r}}n_{j}^{\frac{1}{r}}-n_{j}\right]x_{j} \\
=&\left[\left(1+\frac{1}{n_{i}-1}\right)^{\frac{1}{r}}-\frac{1}{n_{i}-1}-2+\left(1-\frac{1}{n_{j}+1}\right)^{\frac{1}{r}-1}\right]\cdot n_{j}x_{j}\\
\geq&\left[\left(1+\frac{1}{n_{i}-1}\right)^{\frac{1}{r}}+\left(1-\frac{1}{n_{i}-1}\right)^{\frac{1}{r}-1}-\frac{1}{n_{i}-1}-2\right]\cdot n_{j}x_{j}.
\end{split}
\end{displaymath}

We consider the function $f(\xi)=(1+\xi)^{\frac{1}{r}}+(1-\xi)^{\frac{1}{r}-1}-\xi-2$, where $\xi\in [0,\frac{1}{2}]$.
By calculation, we have
$$f'(\xi)=\frac{1}{r}(1+\xi)^{\frac{1}{r}-1}+(\frac{1}{r}-1)(-1)(1-\xi)^{\frac{1}{r}-2}-1,$$
$$f''(\xi)=\frac{1}{r}(\frac{1}{r}-1)(1+\xi)^{\frac{1}{r}-2}+(\frac{1}{r}-1)(\frac{1}{r}-2)(1-\xi)^{\frac{1}{r}-3}.$$
Clearly, $f''(\xi)$ is a strictly increasing function. For $\xi \in (0,\frac{1}{2}]$,
we have $f''(\xi)\geq f''(0)=(\frac{1}{r}-1)(\frac{2}{r}-2)>0$, and then $f'(\xi)>f'(0)=0$,
and furthermore $f(\xi)>f(0)=0$.
Since $n_{i}\geq3$, we have $\frac{1}{n_{i}-1}\in (0,\frac{1}{2}]$. Then
\begin{displaymath}
f\left(\frac{1}{n_{i}-1}\right)=\left(1+\frac{1}{n_{i}-1}\right)^{\frac{1}{r}}+\left(1-\frac{1}{n_{i}-1}\right)^{\frac{1}{r}-1}-\frac{1}{n_{i}-1}-2>f(0)=0.
\end{displaymath}
Therefore the inequality (6) holds.
\end{proof}

\begin{lemma}\label{4'}
Let $G$ be a connected $r$-uniform hypergraph of order $t$, and let $(n_{1}, n_{2},\ldots,n_{t})$ be
a vector of positive integers with $\sum_{i=1}^t n_i=n$. Let $S$ be a subset of $V(G)$ with at least two elements satisfies
$\{i,j\}\subseteq e$ for some $e\in E(G)$ and $L_{G-j}(i)=L_{G-i}(j)$ for any two distinct vertices $i,j\in S$.
In $\mathcal{B}_{n}(G)$, we have

{\em (1)}  if $G\circ (n_{1}, n_{2},\ldots,n_{t})$ attains the maximum spectral radius, then $|n_i-n_j|\le 1$ for any two distinct vertices $i,j\in S$.

{\em (2)}  if $G\circ (n_{1}, n_{2},\ldots,n_{t})$ attains the minimum spectral radius (write $n'=\sum_{i\in S} n_i$),
then there exists a vertex $j\in S$ such that $n_j=n'-|S|+1$ and $n_i=1$ for each $i\in S\setminus \{j\}$.
\end{lemma}
\begin{proof}
(1) Assume $G\circ (n_{1}, n_{2},\ldots,n_{t})$ attains the maximum spectral radius in $\mathcal{B}_{n}(G)$.
Suppose to the contrary that there exist two distinct vertices $i,j\in S$ such that $n_{i}-n_{j}\geq 2$. By Lemma~\ref{4},
$$\rho(G\circ(n_{1},\ldots,n_{i}-1,\ldots,n_{j}+1,\ldots,n_{t}))>\rho(G\circ(n_{1},\ldots,n_{i},\ldots,n_{j},\ldots,n_{t}))$$
contradicting the choice of $G\circ (n_{1}, n_{2},\ldots,n_{t})$. Therefore $|n_i-n_j|\le 1$ for any two distinct vertices $i,j\in S$.

(2)  Assume $G\circ (n_{1}, n_{2},\ldots,n_{t})$ attains the minimum spectral radius in $\mathcal{B}_{n}(G)$.
Suppose to the contrary that there exist two distinct vertices $i,j\in S$ such that $n_{j}\ge n_{i}\geq 2$.
Without loss of generality, assume $i<j$.
Note that $(n_j+1)-(n_i-1)\ge 2$. By Lemma~\ref{4},
$$\rho(G\circ(n_{1},\ldots,n_{i},\ldots,n_{j},\ldots,n_{t}))>\rho(G\circ(n_{1},\ldots,n_{i}-1,\ldots,n_{j}+1,\ldots,n_{t}))$$
contradicting the choice of $G\circ(n_{1},\ldots,n_{i},\ldots,n_{j},\ldots,n_{t})$.
Thus there exists a vertex $j\in S$ such that $n_j=n'-|S|+1$ and $n_i=1$ for each $i\in S\setminus \{j\}$.
\end{proof}

\begin{theorem}\label{5}
For any $H\in \mathcal{B}_{n}(K_{t}^{r})$,
$$\rho(K_{t}^{r}\circ(n-t+1,1,1,\ldots,1))\leq\rho(H)\leq \rho (T_{t}^{r}(n)),$$
where the left equality holds if and only if $H\cong K_{t}^{r}\circ(n-t+1,1,1,\ldots,1)$, and the right equality holds if and only if  $H\cong T_{t}^{r}(n)$.
\end{theorem}
\begin{proof}
Note that a complete $r$-uniform hypergraph $K_{t}^{r}$ satisfies $L_{G-j}(i)=L_{G-i}(j)$, for any two distinct vertices $i,j\in V(K_{t}^{r})$.

On the one hand, assume $H=K_{t}^{r}\circ(n_{1},n_{2},\ldots,n_{t})$ is an $r$-uniform hypergraph
with maximum spectral radius in $\mathcal{B}_{n}(K_t^r)$.
By Lemma \ref{4'}, $|n_{i}-n_{j}|\leq 1$ for any two distinct vertices $i,j\in [t]$ which implies that each $n_{i}$ is equal to either $\lfloor\frac{n}{t}\rfloor$ or $\lceil\frac{n}{t}\rceil$. Thus $H\cong T_{t}^{r}(n)$.

On the other hand, suppose  $H=K_{t}^{r}\circ(n_{1},n_{2},\ldots,n_{t})$ is an $r$-uniform hypergraph
with minimum spectral radius  in $\mathcal{B}_{n}(K_t^r)$.
By Lemma \ref{4'}, there exists a vertex $j\in S$ such that $n_j=n-t+1$ and $n_i=1$ for each $i\in [t]\setminus \{j\}$.
Therefore $H\cong K_{t}^{r}\circ(n-t+1,1,1,\ldots,1)$.
\end{proof}

Combining Theorem \ref{5} and Lemma \ref{2}, we have  the following corollary, which gives  the maximum spectral radius in $\mathcal{B}_{n}(\mathcal{G}^{r}(t))$,  where $\mathcal{G}^{r}(t)$ denotes the set of all  $r$-uniform hypergraphs of order $t$.
This result can also be obtained from \cite{L}, when $p=r$.

\begin{corollary}
For any $H\in \mathcal{B}_{n}(\mathcal{G}^{r}(t))$,
$$\rho(H)\leq \rho (T_{t}^{r}(n)),$$
and the equality holds if and only if $H\cong T_{t}^{r}(n)$.
\end{corollary}
\begin{proof}
Let $H\in \mathcal{B}_{n}(\mathcal{G}^{r}(t))$ and let $G\in\mathcal{G}^{r}(t)$ such that
$H=G\circ(n_{1},n_{2},\ldots,n_{t})$ where $\sum_{i=1}^t n_i=n$. Since $G$ is a subhypergraph of $K_t^r$,
$H$ is a subhypergraph of $K_t^r \circ(n_{1},n_{2},\ldots,n_{t})$.
By Lemma \ref{2},
$$\rho(H)\leq \rho(K_t^r \circ(n_{1},n_{2},\ldots,n_{t})),$$
and the equality holds if and only if $H\cong K_t^r \circ(n_{1},n_{2},\ldots,n_{t})$.
By Theorem \ref{5},
$$\rho(K_t^r \circ(n_{1},n_{2},\ldots,n_{t}))\leq \rho (T_{t}^{r}(n)),$$
and the equality holds if and only if  $K_t^r \circ(n_{1},n_{2},\ldots,n_{t})\cong T_{t}^{r}(n)$.
Therefore
$$\rho(H)\leq \rho (T_{t}^{r}(n)),$$
and the equality holds if and only if $H\cong T_{t}^{r}(n)$.
\end{proof}

\section{Extremal spectral radius of hypergraphs in $\mathcal{B}_{n}(SH(m,q,r))$}

Let $SH(m,q,r)$ be a sunflower hypergraph. For convenience, we label the vertices of $SH(m,q,r)$ of order $t$, which implies $|V(SH(m,q,r))|=r+(m-1)q=t$.
Let $V(SH(m,q,r))=X\cup Y_{1}\cup Y_{2}\cup\cdots\cup Y_{m}$, where $X=\{1,2,\ldots,r-q\}, Y_{1}=\{r-q+1,r-q+2,\ldots,r\}, Y_{2}=\{r+1,r+2,\ldots,r+q\},\ldots,Y_{m}=\{r+(m-2)q+1,r+(m-2)q+2,\ldots,r+(m-1)q=t\}$.
 If  $q=1$, then $SH(m,1,r)\cong K_{r}^{r}\circ(1,1,\ldots,1,m)$.

\begin{theorem}
Let $SH(m,1,r)$ be a sunflower hypergraph of order $t$, and let $(n_{1},n_{2},\ldots,n_{t})$ be  a vector of positive integers
with $\sum_{i=1}^{t}n_{i}=n$. In $\mathcal{B}_{n}(SH(m,1,r))$, $SH(m,1,r)\circ(n_{1},n_{2},\ldots,n_{t})$ attains the minimum spectral radius if and only if $n_{i}=1$ for each $i\in [r-1]$ and $n_{r}+n_{r+1}+\cdots+n_{r+(m-1)}=n-r+1$. Moreover,

{\em (1)}  If $n< mr$, then $SH(m,1,r)\circ(n_{1},n_{2},\ldots,n_{t})$  attains the  maximum spectral radius in $\mathcal{B}_{n}(SH(m,1,r))$ if and only if  $n_{i}=\lfloor\frac{n-m}{r-1}\rfloor$ or $\lceil\frac{n-m}{r-1}\rceil$ for each $i\in [r-1]$, and $n_j=1$ for each $j\in [t]\setminus [r-1]$.

{\em (2)}  If $n\geq mr$, then $SH(m,1,r)\circ(n_{1},n_{2},\ldots,n_{t})$  attains the  maximum spectral radius in $\mathcal{B}_{n}(SH(m,1,r))$ if and only if $SH(m,1,r)\circ(n_{1},n_{2},\ldots,n_{t})\cong  T_{r}^{r}(n)$.
\end{theorem}
\begin{proof}
Note that $(K_{r}^{r}\circ(1,1,\ldots,1,m))\circ(n_{1},n_{2},\ldots,n_{t})=K_{r}^{r}\circ(n_{1},n_{2},\ldots,n_{r-1},n_{r}+\cdots+n_{r+(m-1)})$. By Theorem \ref{5}, $SH(m,1,r)\circ(n_{1},n_{2},\ldots,n_{t})$  attains the minimum spectral radius if and only if $n_{r}+n_{r+1}+\cdots+n_{r+(m-1)}=n-r+1$, $n_{i}=1$ for each $i\in [r-1]$.

If $n\ge mr$, then the conclusion that $SH(m,1,r)\circ(n_{1},n_{2},\ldots,n_{t})$  attains the  maximum spectral radius  if and only if $SH(m,1,r)\circ(n_{1},n_{2},\ldots,n_{t})\cong  T_{r}^{r}(n)$ follows from Theorem \ref{5}.
In the following assume $n<mr$. Note that $K_{r}^{r}$ satisfies $L_{G-j}(i)=L_{G-i}(j)$ for any two distinct $i,j\in V(K_{r}^{r})$ and $n_{r-1}+n_{r}+\cdots+n_{r+(m-1)}\ge m$. By Lemmas~\ref{4} and~\ref{4'}, $SH(m,1,r)\circ(n_{1},n_{2},\ldots,n_{t})$ attains the  maximum spectral radius if and only if $|n_i-n_j|\le 1$ ($n_i\le m$ since $n<mr$) for any $i,j\in [r-1]$ and $n_{r-1}+n_{r}+\cdots+n_{r+(m-1)}=m$.
\end{proof}

\begin{theorem}\label{6}
Let $SH(m,q,r)$ be a sunflower hypergraph of order $t$, and let $(n_{1},n_{2},\ldots,n_{t})$ be  a vector of positive integers. Then
$$\rho(SH(m,q,r)\circ(n_{1},n_{2},\ldots,n_{t}))=\left(\prod_{k\in X}n_{k}\right)^{\frac{r-1}{r}}\cdot\left[\sum\limits_{l=1}^{m}\left(\prod_{k_{l}\in Y_{l}}n_{k_{l}}\right)^{\frac{r-1}{r-q}}\right]^{\frac{r-q}{r}}.$$
\end{theorem}
\begin{proof}
We denote  $SH(m,q,r)\circ(n_{1},n_{2},\ldots,n_{t})$ and $\rho(SH(m,q,r)\circ(n_{1},n_{2},\ldots,n_{t}))$ by $\hat{G}$ and  $\rho$, respectively.
By definition of blow-up hypergraphs,  $\hat{G}$ has vertex partition: $V(\hat{G})=V_{1}\cup V_{2}\cup \cdots\cup V_{t}$,
where $|V_{j}|=n_{j}$ for each $j\in [t]$, and if $\{j_{1},j_{2},\ldots,j_{r}\}\in E(SH(m,q,r))$, then $\{v_{i_1},v_{i_2},\ldots,v_{i_r}\}\in E(\hat{G})$ for every $v_{i_{1}}\in V_{j_{1}}, v_{i_{2}}\in V_{j_{2}},\ldots, v_{i_{r}}\in V_{j_{r}}$.
Let $\mathbf{x}$ be the Perron vector of $\hat{G}$ with $\sum_{v\in V(\hat{G})} x_v^r=1$. Similar to the discussion in Lemma \ref{4},
let $x_{u}=x_{k}$ for any vertex $u\in V_{k}$ where $k\in [t]$.
 By equation (1),   we have
  \begin{displaymath}\tag{8}
\begin{split}
\rho=\mathcal{A}(\hat{G})\mathbf{x}^{r}=\sum_{e\in E(\hat{G})}rx^{e}=r\prod\limits_{k\in X}n_{k}x_{k}\cdot \sum\limits_{l=1}^{m}\prod\limits_{k_{l}\in Y_{l}}n_{k_{l}}x_{k_{l}}.
\end{split}
\end{displaymath}

By the eigenvector equation $\mathcal{A}(\hat{G})\mathbf{x}^{r-1}=\rho \mathbf{x}^{[r-1]}$, from equation (2), we have for each $v_{i}\in V_{i},i\in X$,
 \begin{equation}\tag{9}
 \begin{split}
\rho x_{v_{i}}^{r-1}=\rho x_{i}^{r-1}=\sum_{e\in E_{\hat{G}}(v_{i})} x^{e\setminus\{v_{i}\}}=\prod_{k\in X\setminus
\{i\}}n_{k}x_{k}\cdot \sum\limits_{l=1}^{m}\prod\limits_{k_{l}\in Y_l}n_{k_{l}}x_{k_{l}},
 \end{split}
\end{equation}
and for each $v_{j_{l}}\in V_{j_l},j_l\in Y_l$, $l\in [m]$,
\begin{equation}\tag{10}
 \begin{split}
\rho x_{v_{j_{l}}}^{r-1}=\rho x_{j_{l}}^{r-1}=\sum_{e\in E_{\hat{G}}(v_{j_l})} x^{e\setminus\{v_{j_l}\}}=\prod\limits_{k\in X}n_{k}x_{k}\cdot \prod\limits_{k_{l}\in Y_{l}\setminus\{j_{l}\}}n_{k_{l}}x_{k_{l}}.
  \end{split}
\end{equation}

   Dividing equation (9) by equation (8), we have $x_{i}^{r-1}=\frac{1}{rn_{i}x_{i}}$, so $x_{i}=\left(\frac{1}{rn_{i}}\right)^{\frac{1}{r}}$, where $i\in X$. Then
\begin{displaymath}\tag{11}
\begin{split}
\prod\limits_{k\in X}n_{k}x_{k}=r^{\frac{q-r}{r}}\left(\prod\limits_{k\in X}n_{k}\right)^{\frac{r-1}{r}}.
\end{split}
\end{displaymath}
Next, for all vertices of $Y_{l}$, where $l\in [m]$, we have $q$ equations by (10). Multiplying these $q$ equations, we have
 \begin{displaymath}\tag{12}
\begin{split}
\rho^{q}\prod_{j_{l}\in Y_{l}}x_{j_{l}}^{r-1}&=\left(\prod\limits_{k\in X}n_{k}x_{k}\right)^{q}\cdot \prod\limits_{j_l\in Y_{l}} \left(\prod\limits_{k_{l}\in Y_{l}\setminus\{j_{l}\}}n_{k_{l}}x_{k_{l}}\right)\\
&=\left(\prod\limits_{k\in X}n_{k}x_{k}\right)^{q}\cdot\prod_{j_{l\in Y_{l}}}n_{j_{l}}^{q-1}x_{j_{l}}^{q-1}.
\end{split}
\end{displaymath}
Dividing both sides of equation (12) by $\prod_{j_{l\in Y_{l}}}x_{j_{l}}^{q-1}$,  by combining equation (11), we have
\begin{equation}\tag{13}
\begin{split}
\prod_{j_{l}\in Y_{l}}x_{j_{l}}=\rho^{\frac{-q}{r-q}}r^{\frac{-q}{r}}\left(\prod\limits_{k\in X}n_{k}\right)^{\frac{q(r-1)} {r(r-q)}}\cdot\left(\prod_{j_{l\in Y_{l}}}n_{j_{l}}\right)^{\frac{q-1}{r-q}}.
\end{split}
\end{equation}
Applying equations (11) and  (13) to equation (8), we have
\begin{displaymath}
\begin{split}
\rho=\left(\prod_{k\in X}n_{k}\right)^{\frac{r-1}{r}}\cdot\left[\sum\limits_{l=1}^{m}\left(\prod_{k_{l}\in Y_{l}}n_{k_{l}}\right)^{\frac{r-1}{r-q}}\right]^{\frac{r-q}{r}}.
\end{split}
\end{displaymath}
\end{proof}

In order to find the hypergraphs with extremal spectral radius in $\mathcal{B}_{n}(SH(m,q,r))$, $q\geq2$, we will need  the following results.

\begin{lemma}\label{7}
Let $l$ and $\theta$ be two positive integers with $\theta\ge l \ge 3$.
Let $\mathbf{b}=(b_{1},b_{2},\ldots,b_{l})$ be a vector of positive integers with $\sum_{i=1}^l b_i=\theta$.
We define
$$R(\mathbf{b})=b_{1}^{\beta}(b_{2}^{\beta}+b_{3}^{\beta}+\cdots+b_{l}^{\beta}),
$$
where $\beta$ is any positive real number
no less than 1. Then
$R(\mathbf{b})$ attains the minimum value if and only if  $b_{1}=1$, $b_{i}=\left\lfloor\frac{\theta-1}{l-1}\right\rfloor$ or $\left\lceil\frac{\theta-1}{l-1}\right\rceil$ for $i\neq1$.
\end{lemma}
\begin{proof}
Let $R(\mathbf{b})$ attains its minimum at $\mathbf{b}=(b_{1},b_{2},\ldots,b_{l})$. Now we prove $b_{1}\leq b_2$.
Suppose to the contrary that $b_{1}> b_2$, we consider that a new vector $\mathbf{b'}=(b_{2},b_{1},b_{3},\ldots,b_{l})$, and $\mathbf{b'}$ also satisfy $b_{2}+b_{1}+\cdots+b_{l}=\theta$. Then
\begin{displaymath}
R(\mathbf{b})-R(\mathbf{b'})=(b_{1}^{\beta}-b_{2}^{\beta})(b_{3}^{\beta}+b_{4}^{\beta}+\cdots+b_{l}^{\beta})>0.
\end{displaymath}
This is a contradiction to the choice of $\mathbf{b}$. Hence, $b_{1}\leq b_2$.

Suppose to the contrary that $2 \leq b_{1}\leq b_{2}$.  We consider $\mathbf{b''}=(b_{1}-1,b_{2}+1,b_{3},\ldots,b_{l})$
where $\mathbf{b''}$ also satisfy $b_{1}-1+b_{2}+1+\cdots+b_{l}=\theta$. We have
\begin{displaymath}
\begin{split}
R(\mathbf{b})-R(\mathbf{b''})=&[b_{1}^{\beta}b_{2}^{\beta}-(b_{1}-1)^{\beta}(b_{2}+1)^{\beta}]\\
&+[(b_{1}^{\beta}-(b_{1}-1)^{\beta}]\cdot(b_{3}^{\beta}+b_{4}^{\beta}+\cdots+b_{l}^{\beta})\\
>&0,
\end{split}
\end{displaymath}
which contradicts the choice of $\mathbf{b}$. Hence, $b_{1}=1$ and then $R(\mathbf{b})=b_{2}^{\beta}+b_{3}^{\beta}+\cdots+b_{l}^{\beta}$ with $\sum_{i=1}^2 b_i=\theta-1$.
Since $\beta\ge 1$,  the result follows from the lower convex function property.
\end{proof}

Let $p$ be a positive integer no less than 2. For any positive integer $s\ge p$, define
$$g_p(s)=a^{p-l}(a+1)^{l}$$
where $s=ap+l$, $a=\lfloor\frac{s}{p}\rfloor$ and $0\le l\le p-1$.

\begin{lemma}\label{8}
Let $q$, $m$ and $\theta$ be three positive integers with $q\ge 2$, $m\ge 2$ and $\theta\ge mq$.
Let $\mathbf{s}=(s_{1},s_{2},\ldots,s_{m})$ be a vector of positive integers with $s_m\ge s_{m-1}\ge\cdots \ge s_1\ge q$ and $\sum\limits_{i=1}^m s_i= \theta$.
We define
$$f(\mathbf{s})=g_q^{\beta}(s_1)+g_q^{\beta}(s_2)+\cdots+g_q^{\beta}(s_m)$$
where $\beta$ is any positive real number larger than 1. Then
$f(\mathbf{s})$ attains the maximum value if and only if $s_{1}=s_2=\cdots=s_{m-1}=q$ and $s_m=\theta-(m-1)q$.
\end{lemma}
\begin{proof}
We show if $q+1\le s_i\le s_j$ for some $1\le i<j\le m$, then $f(\mathbf{s})<f(\mathbf{s'})$,
where $\mathbf{s'}=(s_1,\ldots, s_i-1,\ldots,s_j+1,\ldots,s_m)$. Let $s_i=a_iq+r_i$ where $a_i=\lfloor\frac{s_i}{q}\rfloor$ and $0\le r_i\le q-1$ for each $i\in [m]$. Without loss of generality, we only prove the case that $i=1$ and $j=2$. Note that $q+1\le s_1\le s_2$, $\mathbf{s'}=(s_1-1,s_2+1,s_3,\ldots,s_m)$ and $\beta>1$.

If $r_1=0$ and $r_2=0$, then
\begin{displaymath}
\begin{split}
f(\mathbf{s'})-f(\mathbf{s})&=g_q^{\beta}(s_1-1)+g_q^{\beta}(s_2+1)-g_q^{\beta}(s_1)-g_q^{\beta}(s_2)\\
&=(a_1-1)^{\beta} a_1^{\beta(q-1)}+a_2^{\beta(q-1)}(a_2+1)^{\beta}-a_1^{\beta q}-a_2^{\beta q}\\
&=a_2^{\beta(q-1)}[(a_2+1)^{\beta}-a_2^{\beta}]-a_1^{\beta(q-1)}[(a_1)^{\beta}-(a_1-1)^{\beta}]\\
&\ge a_2^{\beta(q-1)}\{[(a_2+1)^{\beta}-a_2^{\beta}]-[(a_1)^{\beta}-(a_1-1)^{\beta}]\}\\
&>0.
\end{split}
\end{displaymath}

If $r_1=0$ and $r_2>0$, then
\begin{displaymath}
\begin{split}
f(\mathbf{s'})-f(\mathbf{s})&=(a_1-1)^{\beta}a_1^{\beta(q-1)}+a_2^{\beta(q-r_2-1)}(a_2+1)^{\beta(r_2+1)}-a_1^{\beta q}-a_2^{\beta (q-r_2)} (a_2+1)^{\beta r_2}\\
&=a_2^{\beta(q-r_2-1)}(a_2+1)^{\beta r_2}[(a_2+1)^{\beta}-a_2^{\beta}]-a_1^{\beta(q-1)}[(a_1)^{\beta}-(a_1-1)^{\beta}]\\
&> a_1^{\beta(q-1)}\{[(a_2+1)^{\beta}-a_2^{\beta}]-[(a_1)^{\beta}-(a_1-1)^{\beta}]\}\\
&> 0.
\end{split}
\end{displaymath}

If $r_1>0$ and $r_2=0$, then $1\le a_1\le a_2-1$ and
\begin{displaymath}
\begin{split}
f(\mathbf{s'})-f(\mathbf{s})&=a_1^{\beta(q-r_1+1)}(a_1+1)^{\beta(r_1-1)}+(a_2+1)^{\beta}a_2^{\beta(q-1)}-a_1^{\beta (q-r_1)} (a_1+1)^{\beta r_1}-a_2^{\beta q}\\
&=a_2^{\beta(q-1)}[(a_2+1)^{\beta}-a_2^{\beta}]-a_1^{\beta(q-r_1)}(a_1+1)^{\beta (r_1-1)}[(a_1+1)^{\beta}-(a_1)^{\beta}]\\
&> a_2^{\beta(q-1)}\{[(a_2+1)^{\beta}-a_2^{\beta}]-[(a_1+1)^{\beta}-(a_1)^{\beta}]\}\\
&> 0.
\end{split}
\end{displaymath}

If $r_1>0$ and $r_2>0$ (we have $r_1\le r_2$ when $a_1=a_2$), then
\begin{displaymath}
\begin{split}
&f(\mathbf{s'})-f(\mathbf{s})\\
=&a_1^{\beta(q-r_1+1)}(a_1+1)^{\beta(r_1-1)}+a_2^{\beta(q-r_2-1)}(a_2+1)^{\beta(r_2+1)}-a_1^{\beta (q-r_1)} (a_1+1)^{\beta r_1}-a_2^{\beta (q-r_2)} (a_2+1)^{\beta r_2}\\
=&a_2^{\beta(q-r_2-1)}(a_2+1)^{\beta r_2}[(a_2+1)^{\beta}-a_2^{\beta}]-a_1^{\beta(q-r_1)}(a_1+1)^{\beta (r_1-1)}[(a_1+1)^{\beta}-a_1^{\beta}]\\
>&a_1^{\beta(q-r_1)}(a_1+1)^{\beta (r_1-1)}\{[(a_2+1)^{\beta}-a_2^{\beta}]-[(a_1+1)^{\beta}-a_1^{\beta}]\}\\
\ge& 0.
\end{split}
\end{displaymath}

By the above discussion, we have $f(\mathbf{s})$ attains the maximum value if and only if $s_{1}=s_2=\cdots=s_{m-1}=q$ and $s_m=\theta-(m-1)q$.
\end{proof}

\begin{theorem}\label{9}
Let $q$, $m$, $r$, $t$ and $n$ be positive integers with $q\ge 2$, $m\ge 2$, $r\ge 3$ and $n\ge t= r+(m-1)q$.
Let $p_1,p_2,\ldots,p_c$ be all the maximum points of
\begin{equation}\tag{14}\label{e14}
[g_{r-q}(s)]^{\frac{r-1}{r-q}}\cdot\left[m-1+ [g_q(n-s-(m-1)q)]^{\frac{r-1}{r-q}}\right]
\end{equation}
where $r-q\le s \le n-mq$. For each $k\in [c]$,
let $H_k=SH(m,q,r)\circ (n_{1},n_{2},\ldots,n_{t})$ where $\sum_{i\in X} n_i=p_k$, $\sum_{j\in Y_m} n_j=n-(m-1)q-p_k$, $|n_i-n_j|\le 1$
for any two distinct vertices $i,j\in X$ or $i,j\in Y_m$, and $n_j=1$ for each $j\in \bigcup_{l=1}^{m-1} Y_l$.
In $\mathcal{B}_{n}(SH(m,q,r))$,

{\em (1)} $SH(m,q,r)\circ (n_{1},n_{2},\ldots,n_{t})$ attains the minimum spectral radius if and only if $n_{r},n_{r+q},\ldots,n_{r+(m-1)q}=\left\lfloor\frac{n-t}{m}\right\rfloor+1$ or $\left\lceil\frac{n-t}{m}\right\rceil+1$
and $n_{i}=1$ for each $i\in [t]\setminus \{r, r+q,\ldots,r+(m-1)q\}$.

{\em (2)} $H=SH(m,q,r)\circ (n_{1},n_{2},\ldots,n_{t})$ attains the maximum spectral radius if and only if
$H$ is isomorphic to one of $H_1,H_2,\ldots, H_c$.
\end{theorem}
\begin{proof}
Let $G=SH(m,q,r)$. Assume $\tilde{G}=G\circ(n_{1},n_{2},\ldots,n_{t})$ attains minimum spectral radius in $\mathcal{B}_{n}(SH(m,q,r))$.
For any two distinct vertices $i,j\in X$ or $i,j\in Y_l$ where $l\in [m]$, $E_G(i)=E_G(j)$ implies $L_{G-j}(i)=L_{G-i}(j)=\emptyset$.
By Lemma \ref{4'}, without loss of generality, we may assume $n_{2}=n_{3}=\cdots=n_{r-q}=1$ and $n_{k}=1$ for each $k\in (Y_{1}\cup Y_{2}\cup\cdots\cup Y_{m})\setminus \{r,r+q,\ldots,r+(m-1)q\}$.
By Theorem \ref{6},
$$\rho(\tilde{G})=n_{1}^{\frac{r-1}{r}}\cdot\left[\sum\limits_{l=0}^{m-1}\left(n_{r+lq}\right)^{\frac{r-1}{r-q}}\right]^{\frac{r-q}{r}}=\left[n_{1}^{\frac{r-1}{r-q}}\cdot\sum\limits_{l=0}^{m-1}\left(n_{r+lq}\right)^{\frac{r-1}{r-q}}\right]^{\frac{r-q}{r}}.$$
Take $\mathbf{b}=(n_{1},n_{r},n_{r+q},\ldots,n_{r+(m-1)q})$ and $\beta=\frac{r-1}{r-q}\ge 1$ in Lemma \ref{7},
the result in (1) holds.

Assume $H=G\circ(n_{1},n_{2},\ldots,n_{t})$ attains maximum spectral radius in $\mathcal{B}_{n}(SH(m,q,r))$.
For any two distinct vertices $i,j\in X$ or $i,j\in Y_l$ where $l\in [m]$, $E_G(i)=E_G(j)$ implies $L_{G-j}(i)=L_{G-i}(j)=\emptyset$.
By Lemma \ref{4'}, $|n_i-n_j|\le 1$. Let $s=\sum_{i\in X} n_i$ and $s_l=\sum_{j\in Y_l} n_j$ for each $l\in [m]$.
Without loss of generality assume $s_1\le s_2 \le \cdots \le s_m$.
By Theorem \ref{6},
\begin{displaymath}
\begin{split}
\rho(H)&=\left(\prod_{k\in X}n_{k}\right)^{\frac{r-1}{r}}\cdot\left[\sum\limits_{l=1}^{m}\left(\prod_{k_{l}\in Y_{l}}n_{k_{l}}\right)^{\frac{r-1}{r-q}}\right]^{\frac{r-q}{r}}\\
&=[g_{r-q}(s)]^{\frac{r-1}{r}}\cdot\left[\sum\limits_{l=1}^{m} [g_q(s_l)]^{\frac{r-1}{r-q}}\right]^{\frac{r-q}{r}}.
\end{split}
\end{displaymath}

Take $\theta=n-s$ and $\beta=\frac{r-1}{r-q}$ in Lemma~\ref{8}, we have $s_{1}=s_2=\cdots=s_{m-1}=q$ and $s_m=n-s-(m-1)q$.
So $\rho(H)$ equals to the maximum value of $[g_{r-q}(s)]^{\frac{r-1}{r}}\cdot\left[m-1+ [g_q(n-s-(m-1)q)]^{\frac{r-1}{r-q}}\right]^{\frac{r-q}{r}}$ where $r-q\le s \le n-mq$.
Thus $H=SH(m,q,r)\circ (n_{1},n_{2},\ldots,n_{t})$ attains the maximum spectral radius if and only if
$H$ is isomorphic to one of $H_1,H_2,\ldots, H_c$.
\end{proof}

By Theorem~\ref{9}, we see that if the blow-up of $SH(m,q,r)$ attains the maximum spectral radius in $\mathcal{B}_{n}(SH(m,q,r))$,
then we only need to blow-up any one edge. Since there are 4 unknown parameters $n,m,q,r$, it
is difficult to determine the maximum points of function~(\ref{e14}).


\end{document}